\documentclass[12pt,dvips,twoside,letterpaper]{article}
\usepackage{pslatex}
\usepackage{fancyhdr}
\usepackage{graphicx}
\usepackage{geometry}
\RequirePackage[latin1]{inputenc}
\RequirePackage[T1]{fontenc}

\def\figurename{Figure} 
\makeatletter
\renewcommand{\fnum@figure}[1]{\figurename~\thefigure.}
\makeatother
\def\tablename{Table} 

\makeatletter
\renewcommand{\fnum@table}[1]{\tablename~\thetable.}
\makeatother
\usepackage{color}
                [2000/03/29 v1.0m Standard LaTeX package]

\usepackage{bbm}
\usepackage{amsmath}
\usepackage{amssymb}
\usepackage{amsfonts}
\usepackage{amsthm,amscd}
\newtheorem{theorem}{Theorem}[section]

\newtheorem{corollary}[theorem]{Corollary}
\newtheorem{proposition}[theorem]{Proposition}
\theoremstyle{definition}

\newtheorem{definition}[theorem]{Definition}

\theoremstyle{remark}
\newtheorem{remark}[theorem]{Remark}

\numberwithin{equation}{section}

\def\P{\mathbb P}
\def\R{\mathbb R}
\def\E{\mathbb E}
\def\F{\mathbb F}

\def\Q{\mathbb Q}

\def\E{\mathbb E}

\begin{document}

\title{Robust optimal problem for dynamic risk measures governed by BSDEs with jumps and delayed generator}

\author{Navegué Tuo\thanks{tnavegue@yahoo.fr}\; and Auguste Aman\thanks{augusteaman5@yahoo.fr, corresponding author}\\
UFR de Mathématiques et Informatique\\ Université Félix H. Boigny, Cocody\\
22 BP 582 Abidjan, Côte d'Ivoire}

\date{}
\maketitle \thispagestyle{empty} \setcounter{page}{1}

\thispagestyle{fancy} \fancyhead{}
 \fancyfoot{}
\renewcommand{\headrulewidth}{0pt}

\begin{abstract}
The aim of this paper is to study an optimal stopping problem for dynamic risk measures induced by backward stochastic differential equations with jumps and delayed generator. Firstly, we connect the value function of this problem to reflected BSDEs with jump and delayed generator. Furthermore, after establishing existence and uniqueness result for this  reflected BSDE, we use its to address through a mixed/optimal stopping game problem for the previous dynamic risk measure in ambiguity case.
\end{abstract}

\vspace{.08in}\textbf{MSC}:Primary: 60F05, 60H15, 47N10, 93E20; Secondary: 60J60\\ 
\vspace{.08in}\textbf{Keywords}: Backward stochastic differential equations; Delayed generators Reflected backward stochastic equations; Jump processes; Optimal stopping; Dynamic risk measures; Game problems.

\section{Introduction}
The risk measures start with the work of Artzner et al. \cite{Aal}. Later, there has been a lot of studies on risk measures. See e.g Follmer and Shied \cite{E6}, Frittelli and Gianin \cite{E8}, Bion-Nadal \cite{E16}, Barrieu and El Karoui \cite{E21}, Bayraktar E, I. Karatzas and Yao \cite{E18}. After these, around the year 2005, various authors established the links between continuous time dynamic risk measures and
backward differential equations. They have introduced dynamic risk measures in the Brownian case, defined as the solutions of BSDEs (see \cite{E8, E9,E21}). Clearly,
let consider $f$ and $\xi$ respectively a function and random variable.
The risk measure of the position $\xi$ denoted by $\rho_{t}(\xi)$ is described by the process $-X_{t}$ where $\{X(t),\; t\geq 0\}$ is the first component solution of BSDEs
associated to generator $f$  and terminal value $\xi$.
Many studies have been done on such risk measures, dealing with optimal
stopping problem and robust optimization problems (see for example \cite{10,E18,E21}). 

Recently, in \cite{D1}, Delong and Imkeller introduced the theory of nonlinear backward
stochastic differential equations (BSDEs, in short) with time
delayed
generators. Precisely, given a progressively measurable process $f$,
so-called generator and a square integrable random variable $\xi$, BSDEs with
time delayed generator are BSDEs of the form:   
\begin{eqnarray*}
X(t) = \xi+\int_{t}^{T}f(s,X_{s},Z_{s})ds-\int_{t}^{T}Z(s)dW(s), 0 \leq t \leq T,
\end{eqnarray*}
where the process $(X_{t},Z_{t})=(X(t+u),Z(t+u))_{-T \leq u \leq 0}$ represents all the past values of the solution until $t$. Under some assumptions, they
proved existence and uniqueness result of such a BSDEs. In this dynamic, the
same authors study, in an accompanying paper (see \cite{DI2}), BSDE with time
delayed generator driven both by a Brownian motion and a Poisson random
measure. Existence and uniqueness of a solution and its Malliavin's differentiability has been established. A few year later, in \cite{D2}, Delong proved that BSDEs with time delayed generator is a important tool to formulate many problems in mathematical finance and insurance.
For example, he proved that the dynamic of option based portfolio assurance is
the following time delayed BSDE:
\begin{eqnarray*}
X(t)=X(0)+(X(T)-X(0))^{+}-\int_t^TZ(s)dW(s).
\end{eqnarray*}
From these works, and given the importance of applications related to BSDEs
with time delayed generator, in your opinion, it is very judicious to expect to
study an optimal stopping problem for dynamic risk measures governed by
backward stochastic differential equations with delayed generator. Better, this
paper is dedicated to resolve an optimal stopping problem for dynamic
risk measure governed by backward stochastic differential equations driven with
both Brownian motion and Poison random measure. For more detail, let consider
$(\psi(t))_{t\geq 0}$ a given right continuous left limited adapted process and $\tau$ be a stopping time in $[0,T]$. Our
objective is to solve an optimal stopping problem related to risk measure of the
position $\psi(\tau)$ denoted by $\rho^{\psi,\tau}$ with dynamic follows as the process $-X^{\psi,\tau}$ where $(X^{\psi,\tau},Z^{\psi,\tau}, U^{\psi,\tau})$ satisfied the following BSDE
\begin{eqnarray*}
X^{\psi,\tau}(t)&=&\psi(\tau)+\int_t^Tf(s,Z^{\psi,\tau}_{s},U^{\psi,\tau}_{s}(.))ds-\int_t^TZ^{\psi,\tau}(s)dW(s)\\
&&-\int_t^T\int_{\R^* }U^{\psi,\tau}(s,z)\tilde{N}(ds,dz),\;\; 0\leq t\leq \tau,
\end{eqnarray*}
where $\R^*=\R\textbackslash \{0\}$. 

Roughly speaking, for all stopping time $\sigma$ with values on $[0,T]$, our aim is to minimize the risk measures at time $\sigma$ i.e we want to find a unique stopping time $\tau^{*}$ such that setting 
\begin{eqnarray*}
v(\sigma)=ess\inf_{\sigma\leq \tau\leq T}\rho^{\psi,\tau}(\sigma),
\end{eqnarray*}
we have
\begin{eqnarray}\label{ST}
v(\sigma)=\rho^{\psi,\tau^{*}}(\sigma). 
\end{eqnarray}
Our method is essentially based on the link establish between the value function $v$ and the first component of the solution of a reflected BSDEs with jump and delayed generator. Notion of reflected BSDEs has been introduced for the first time by N. EL Karoui et al. in \cite{E14} with a Brownian filtration. The solutions of such equations are constrained to be greater than given continuous processes called obstacles. Later, different extensions have been performed when we add the jumps process and/or suppose the obstacle not continuous. One can cite works of Tang and Li \cite{TL}, Hamadène and Ouknine \cite{E5,E7}, Essaky \cite{Essaky} and Quenez and Sulem \cite{10}. More recently, reflected BSDEs without jump and with delayed generator have been introduced respectively by Zhou and Ren \cite{ZR}, and Tuo et al. \cite{Tal}. 
Our study takes place in two stages. First, we provide an optimality criterium, that is a characterization of optimal stopping times and when the obstacle is right continuous and left limited (rcll, in short), we show the existence of an optimal stopping time. Thereafter, we address the optimal stopping problem when there is ambiguity on the risk measure. It means that there exists a given control $\delta$ that can influence the dynamic risk measures. More precisely, given the dynamic position $\psi$ this situation consists to focus on the robust optimal stopping problem for the family of risk measures $\{\rho^{\delta},\;\; \delta\in\mathcal{A}\}$ of this position $\psi$ induced by the BSDEs associated with generators $\{f^{\delta},\; \delta\in \mathcal{A}\}$. To this purpose and in view of the first part, we study the following optimal control problem related to $Y^{\delta}$ the first component solution of reflected BSDEs with jumps and delayed generator $f^{\delta},\;\;\delta \in A)$ with a RCLL obstacle $\psi$. In other words, we want to determine a stopping time $\tau^{*}$, which minimizes over all stopping times $\tau^{\delta}$, the risk of the position $\psi$. This is equivalent to derive a saddle points to a mixed control/optimal stopping game problem.

  The paper is organized as follows. We give the notation and formulation of the optimal problem for risk measures problem in Section 2. Existence and uniqueness results for RBSDEs with jumps and delayed generator with right continuous left limit (rcll) obstacle is provided in Section 3. In both section 4 and 5, we deal with the robust optimal stopping problem.

\section{Formulation of the problem}
Let consider a probability space $(\Omega,\mathcal{F},\P)$. For $E=\R^{d}\textbackslash \{0\}$ equipped with its Borel field $\mathcal{E}$, let $N$ be a Poisson random measure on $\R_{+}\times E$ with compensator $\nu(dt,dx)=\lambda(dx)dt$ where $\lambda $ is $\sigma$-finite measure on $(E,\mathcal{E})$ satisfying
\begin{eqnarray*}
\int_{E}(1\wedge |x|^2)d\lambda(x)<+\infty.
\end{eqnarray*}
such that $((N-\nu)([0,t]\times A))_{t\geq 0}$ is a martingale. Let also consider ac$d$-dimensional standard Brownian motion $(W_t)_{t\geq 0}$ independent of $N$. Let finally consider the filtration $\mathbb{F}={\mathcal{F}_t}_{t\geq t}$ defined by
\begin{eqnarray*}
\mathcal{F}_t=\mathcal{F}^{W}\wedge\mathcal{F}^{N}\wedge\mathcal{N},
\end{eqnarray*}
where $\mathcal{N}$ is the set of all $\P$-null element of $\mathcal{F}$.

\subsection{BSDEs with time delayed generators driven by Brownian motions and Poisson random measures}

This subsection is devoted to recall existence and uniqueness result for BSDEs with jump and time-delayed generator
\begin{eqnarray}\label{BSDEjump}
X(t)&=&\xi+\int_t^Tf(s,X_s,Z_{s},U_{s}(.))ds-\int_t^TZ(s)dW(s)\nonumber\\
&&-\int_t^T\int_{E}U(s,z)\tilde{N}(ds,dz),\;\; 0\leq t\leq T,
\end{eqnarray}
studied by Delong and Imkeller in \cite{DI2} and derive a comparison principle associated to this BSDE.
In this instance, let us describe following spaces of processes:
\begin{description}
\item $\bullet $ $ L_{-T}^2 (\mathbb{R}) $ denotes the space of measurable functions $ z : [-T,0]\rightarrow\mathbb{R}$ satisfying
$$ \int_{-T}^0 \mid z(v) \mid^2 dv<+\infty,
$$
\item $ \bullet $ $L_{-T}^{\infty} (\mathbb{R} )$ denotes the space of bounded, measurable functions $ y : [-T,0] \rightarrow \mathbb{R} $\\
satisfying
$$
\sup\limits_{v\in [-T,0]}\mid y(v) \mid^2 <+\infty,
$$
\item $L_{-T,m}^{2}(\R)$ denotes the space of product measurable functions $u:[-T,0]\times \R/\{0\}\rightarrow \R$                                                                                                such that
$$
\int_{-T}^{0}\int_{E}|u(t,z)|^{2}m(dz)dt<+\infty.
$$

\item $\bullet$ $ L^2 (\Omega, \mathcal{F}_T,\mathbb{R})$ is the Banach space of $\mathcal{F}_T$-measurable random variables $\xi: \Omega \rightarrow \mathbb{R} $ normed by $\displaystyle \|\xi\|_{L^2}=\left[\E(|\xi|^2)\right]^{1/2}$
\item $\bullet$  $\mathcal{H}^2(\R)$ denotes the Banach space of all predictable processes $\varphi$ with values in $\R$ such that $\E\left[\int_0^T|\varphi(s)|^2ds\right]<+\infty$.
\item $\bullet$ Let $\mathcal{H}_{m}^{2}(\R)$ denote the space of $\mathcal{P}\otimes \mathcal{E}$-mesurable processes $\phi$ satisfying $\E\left(\int_{0}^{T}\int_{E}|\phi(t,z)|^{2}m(dz)dt\right)<+\infty$, where $\mathcal{P}$ is the sigma algebra of $(\mathcal{F}_t)_{t\geq 0}$-predictable set on $\Omega\times[0,T]$.
\item $\bullet$ $\mathcal{S}^2(\R)$ denotes the Banach space of all $(\mathcal{F}_t)_{0\leq t\leq T}$-adapted right continuous left limit (rcll) processes $\eta$ with values in $\R$ such that  $\E\left(\sup_{0\leq s\leq T}|\eta(s)|^2\right)<+\infty$
\item $\bullet$ $\mathcal{K}^2(\R)$ denotes the Banach space of all $(\mathcal{F}_t)_{0\leq t\leq T}$-predictable right continuous left limit (rcll) increasing processes $\eta$ with values in $\R$ such that  $\eta(0)=0$ and $\E\left(|\eta(T)|^2\right)<+\infty$
\end{description}
The spaces $\mathcal{H}^2(\R),\, \mathcal{H}_{m}^{2}(\R)$ and $\mathcal{S}^2(\R)$ are respectively endowed with the norms
\begin{eqnarray*}
\|\varphi\|^2_{\mathcal{H}^2,\beta}&=&\E\left[\int_0^Te^{\beta s}|\varphi(s)|^2ds\right]\\
\|\phi(t,z)\|^2_{\beta,m}&=& \E\left(\int_{0}^{T}\int_{E}|\phi(t,z)|^{2}m(dz)dt\right)\\ 
\|\eta (s)\|^2_{\mathcal{S}^{2},\beta}&=&\E\left(\sup_{0\leq s\leq T}e^{\beta s}|\eta(s)|^2\right). 
\end{eqnarray*}

Our two results has been done under the following hypotheses: For a fix $T>0$,
\begin{description}
\item[$({\bf A1})$] $\tau$ is a finite $(\mathcal{F}_t)_{0\leq t\leq T}$-stopping time.
\item[$({\bf A2})$] $\xi \in L^{2}(\mathcal{F}_{\tau},\R)$
\item[$({\bf A3})$] $f:\Omega \times [0,T]\times L_{-T}^{\infty}(\R) \times L_{-T}^{2}(\R) \times L_{-T,m}^{2}(\R) \rightarrow \R$ is a product measurable, $\F$-adapted function satisfying
\begin{itemize}
	\item [$(i)$] There exists a probability measure $\alpha$ on $([-T,0],\mathcal{B}([-T,0]))$ and a positive constant $K$, such that
\begin{eqnarray*}
&&|f(t,y_t,z_t,u_t(.)- f(t,\bar{y}_t,\bar{z}_t,\bar{u}_t(.)|^{2}\\
&\leq K &\int_{-T}^{0}\left[|y(t+v)-\bar{y}(t+v)|^{2}+|z(t+v)-\bar{z}(t+v)|^{2}\right.\\
&&\left.+\int_{E}|u(t+v,\zeta) - \bar{u}(t+v,\zeta)|^{2}m(d\zeta)\right]\alpha(dv)
\end{eqnarray*}
for $\mathbb{P}\otimes\lambda$ a.e, $(\omega,t) \in \Omega \times [0,T]$, for any $(x_t,z_t,u_t(.))$, $(\bar{x}_t,\bar{z}_t,\bar{u}_t(.))\in L_{-T}^{\infty}(\R) \times L_{-T}^{2}(\R) \times L_{-T,m}^{2}(\R)$ 

\item [$(ii)$] $\displaystyle \E\left(\int_{0}^{T}|f(t,0,0,0)|^{2}dt \right)<+\infty$
\item [$(iii)$] $f(t,.,.,.)= 0$ a.s, for $t<0$
\end{itemize}
\end{description}
For the sake of good understanding, we give in the following the notion of solution of BSDE \eqref{BSDEjump}.
\begin{definition}
 The triple processes $(X,Z,U)$ is called solution of BSDE \eqref{BSDEjump} if $(X,Z,U)$ belongs in $\mathcal{S}^2(\R)\times \mathcal{H}^2(\R)\times \mathcal{H}^2_{m}(\R)$ and satisfies \eqref{BSDEjump}.
\end{definition}
We recall the existence and uniqueness result established in \cite{DI2}.
\begin{theorem}
Assume that $({\bf A1})$-$({\bf A3})$ hold. If $T$ a terminal time or $K$ a Lipschitz constant are sufficiently small i.e
\begin{eqnarray*}
9TKe\max(1,T)<1,
\end{eqnarray*}
\eqref{BSDEjump} has a unique solution.
\end{theorem}
The concept of comparison principle is a very important in the theory of BSDE without delay. Unfortunately, as point out by Example 5.1 in \cite{D1}, this principle cannot be extended in general form to BSDEs with delayed generators. Nevertheless, according to Theorem 3.5 appear in \cite{Tal}, the comparison principle for  BSDEs without jump and with delayed generator, still hold on stochastic intervals in where the strategy process $Z$ stays away from $0$. The following theorem is an extension to BSDEs with jump and delayed generator. To do it, we need this additional assumption
\begin{description}
\item ({\bf A4 }) $f:\Omega \times [0,T]\times L_{-T}^{\infty}(\R) \times L_{-T}^{2}(\R) \times L_{-T,m}^{2}(\R) \rightarrow \R$ is a product measurable, $\F$-adapted function satisfying:\begin{eqnarray*}
f(t,x_t,z_t,u_t(.))-f(t,x_t,z_t,u'_t(.))\geq \int_{-T}^{0}\langle \theta^{x_t,z_t,u_t(.),u'_t(.)},u(t+v,.)-u'(t+v,.)\rangle_m\alpha(dv),
\end{eqnarray*}
 for $\mathbb{P}\otimes\lambda$ a.e, $(\omega,t) \in \Omega \times [0,T]$ and each $(x_t,z_t,u_t(.),u'_t(.))\in L_{-T}^{\infty}(\R) \times L_{-T}^{2}(\R) \times L_{-T,m}^{2}(\R)\times L_{-T,m}^{2}(\R)$, where $ \theta:\Omega\times [0,T]\times L_{-T}^{\infty}(\R) \times L_{-T}^{2}(\R) \times L_{-T,m}^{2}(\R)\times L_{-T,m}^{2}(\R)\rightarrow L_{-T,m}^{2}(\R)$ is 
 a measurable an bounded function such that there exists $\varphi$ belongs to $L_{-T,m}^{2}(\R)$, verifying
 \begin{eqnarray*}
\theta^{x_t,z_t,u_t(.),u'_t(.)}(\zeta)\geq -1\;\;\;\;\mbox{and}\;\;\;\; |\theta^{x_t,z_t,u_t(.),u'_t(.)}(\zeta)|\leq \varphi(\zeta).
\end{eqnarray*}

\end{description}
\begin{theorem}
Consider BSDE \eqref{BSDEjump} associated to delayed generators $f_1$, $f_2$ and corresponding terminal values $\xi^1$, $\xi^2$ at terminal time $\tau$ satisfying the assumptions $({\bf A1})$-$({\bf A3})$. Let $(X^{\tau,1}, Z^{\tau,1},U^{\tau,1})$ and $(X^{\tau,2}, Z^{\tau,2},U^{\tau,2})$ denote respectively the associated unique solutions. Let consider the sequence of stopping time $(\sigma_n)_{n\geq 1}$ define by  
\begin{eqnarray}
\sigma_n&=&\inf\left\{
\begin{array}{ll}
t\geq 0 &,\displaystyle |X^{\tau,1}(t)-X^{\tau, 2}(t)|\vee|Z^{\tau,1}(t)-Z^{\tau,2}(t)|\vee \int_{E}|U^{\tau,1}(t,z)-U^{\tau,2}(t,z)|m(dz) \leq \frac{1}{n}\\ \mbox{or}\\
&\displaystyle |X^{\tau,1}(t)-X^{\tau, 2}(t)|\vee |Z^{\tau,1}(t)-Z^{\tau,2}(t)|\vee \int_{E}|U^{\tau,1}(t,z)-U^{\tau,2}(t,z)|m(dz)\geq n.
\end{array}
\right\}\nonumber
\\ &\wedge & T\label{TA}
\end{eqnarray}
and set 
\begin{eqnarray}\label{TAbis}
\sigma=\sup_{n\geq 1}\sigma_n.
\end{eqnarray}
Moreover we suppose that
\begin{itemize}
\item $X^{\tau,1}(\sigma)\geq X^{\tau,2}(\sigma)$
\item $f_1(t, X^{\tau,1}_t,Z^{\tau,1}_t,U^{\tau,1}_t(.))\geq f_2(t,X^{\tau,1}_t,Z^{\tau,1}_t,U^{\tau,1}_t(.))$ or
\item $f_1(t,X^{\tau,2}_t,Z^{\tau,2}_t,U^{\tau,2}_t(.))\geq f_2(t,X^{\tau,2}_t,Z^{\tau,2}_t,U^{\tau,2}_t(.))$. 
\end{itemize}
Then $X^{\tau,1}(t)\geq X^{\tau,2}(t),\; \P$-a.s. for all $t\in [0,\sigma]$. 
\end{theorem}
\begin{proof}
We follow the ideas from  Theorem 5.1 for BSDEs without jumps and with
delayed generator established in \cite{D1}. For each $t\in [0,T]$ let 
\begin{eqnarray*}
\Delta X^{\tau}(t)=X^{\tau,1}(t)-X^{\tau,2}(t),\;\Delta Z(t)= Z^{\tau,1}(t)- Z^{\tau,2}(t),\;\;\Delta U^{\tau}(t,.)= U^{\tau,1}(t,.)- U^{\tau,2}(t,.),\\\\ \Delta f(t,X^{\tau,2}_{t},Z^{\tau,2}_{t},U^{\tau,2}_{t}(.)) =  f^{1}(t,X^{\tau,2}_{t},Z^{\tau,2}_{t},U^{\tau,2}_{t}(.))-f^{2}(t,X^{\tau,2}_{t},Z^{\tau,2}_{t},U^{\tau,2}_{t}(.)).
\end{eqnarray*}
Let consider the real processes $\delta, \beta$ and $\gamma$ defined respectively by  
\begin{eqnarray*}
\delta(t) =\left\{
\begin{array}{lll}
\frac{f^{1}(t,X^{\tau,1}_{t},Z^{\tau,1}_{t},U^{\tau,1}_{t}(.))- f^{1}(t,X^{\tau,2}_{t},Z^{\tau,1}_{t},U^{\tau,1}_{t}(.))}{\Delta X^{\tau}(t)}&\mbox{if}& \Delta X^{\tau}(t)\neq 0\\
0 & & otherwise,
\end{array}
\right.
\end{eqnarray*}

\begin{eqnarray*}
\beta(t)=
\left\{
\begin{array}{lll}
\frac{f^{1}(t,X^{\tau,2}_{t},Z^{\tau,1}_{t},U^{\tau,1}_{t}(.))- f^{1} (t,X^{\tau,2}_{t},Z^{\tau,2}_{t},U^{\tau,1}_t(.))}{\Delta Z^{\tau}(t)}&\mbox{if}&\Delta Z^{\tau}(t)\neq 0\\
 0 & & otherwise.		
\end{array}
\right.
\end{eqnarray*}
and 
\begin{eqnarray*}
\gamma(t)=
\left\{
\begin{array}{lll}
\frac{f^{1}(t,X^{\tau,2}_{t},Z^{\tau,2}_{t},U^{\tau,1}_{t}(.))- f^{1} (t,X^{\tau,2}_{t},Z^{\tau,2}_{t},U^{\tau,2}_t(.))}{\int_{E}\Delta U^{\tau}(t,z)m(dz)}&\mbox{if}&\int_{E}\Delta U^{\tau}(t,z)m(dz)\neq 0\\
 0 & & otherwise.		
\end{array}
\right.
\end{eqnarray*} 
Hence, since $f^1$ and $f^2$ are Lipschitz with respect $x$, $z$ and in $u$, we have 
\begin{eqnarray*}
|\delta(t)|^2\leq K\int_{-T}^{0}\left(\frac{|\Delta X^{\tau}(t+u)|^2}{|\Delta X^{\tau}(t)|^2}\right)\alpha(du),
\end{eqnarray*}
\begin{eqnarray*}
|\beta(t)|^2\leq K\int_{-T}^{0}\left(\frac{|\Delta Z^{\tau}(t+u)|^2}{|\Delta Z^{\tau}(t)|^2}\right)\alpha(du)
\end{eqnarray*}
and 
\begin{eqnarray*}
|\gamma(t)|^2\leq K\int_{-T}^{0}\left(\frac{\displaystyle \int_{E}|\Delta U^{\tau}(t+u,z)|^2 m(dz)}{\displaystyle\int_{E}|\Delta U^{\tau}(t,z)|^2 m(dz)}\right)\alpha(du).
\end{eqnarray*}
Next, in view of \eqref{TA} and \eqref{TAbis}, for $t\in[0,\sigma]$, there exist a constant $C$ such that $\phi=\delta, \beta, \gamma$,
$$|\phi(t)|\leq C,\;\; a.s.$$ 
On other hand, we have
\begin{eqnarray*}
\Delta X^{\tau}(t)&=& \Delta X^{\tau}(\sigma)+\int_t^{\sigma}\delta(s)\Delta X^{\tau}(s)ds + \int_t^{\sigma}\beta(s)\Delta Z(s)ds\\
&& + \int_t^{\sigma}\int_{E}\gamma(s)\Delta U^{\tau}(t,z)m(dz)ds+\int_{t}^{\sigma}\Delta f(s,X^{\tau,2}_{s},Z^{\tau,2}_{s},U^{\tau,2}_{s}(.))ds\\
&&-\int\Delta Z(s)dW(s) \int_t^{\sigma}\int_{E}\Delta U(s,z)\tilde{N}(ds,dz)
\end{eqnarray*}
and setting $\displaystyle R(t)=\int_{0}^{t}\delta(s)ds$, it follows from Itô's formula applied to $R(s)\Delta X^{\tau}(s)$ between $t$ to $\sigma$ that
\begin{eqnarray*}\label{Comp}
R(t)\Delta X^{\tau}(t) & = & R(\sigma)\Delta X^{\tau}(\sigma)+ \int_{t}^{\sigma} R(s)\beta(s)\Delta Z^{\tau}(s)ds\nonumber\\
&& +\int_{t}^{\sigma}\int_{E}R(s)\gamma(s)\Delta U^{\tau}(t,z)m(dz)ds+\int_{t}^{\sigma} R(s)\Delta f(s,X^{\tau,2}_{s},Z^{\tau,2}_{s},U^{\tau,2}_{s}(.))ds\nonumber\\
&& - \int_{t}^{\sigma} R(s)\Delta Z^{\tau}(s)dW(s)- \int_{t}^{\sigma}\int_{E}R(s)\Delta U^{\tau}(s,z)\tilde{N}(ds,dz).                                                                  
\end{eqnarray*}
Taking into consideration the assumptions on generators and terminal values, we obtain
\begin{eqnarray}\label{Comp}
R(t)\Delta X^{\tau}(t) &\leq &\int_{t}^{\sigma} R(s)\beta(s)\Delta Z^{\tau}(s)ds\nonumber\\
&& +\int_{t}^{\sigma}\int_{E}R(s)\gamma(s)\Delta U^{\tau}(t,z)m(dz)ds\nonumber\\
&& - \int_{t}^{\sigma} R(s)\Delta Z^{\tau}(s)dW(s)-\int_{t}^{\sigma}\int_{E}R(s)\Delta U^{\tau}(s,z)\tilde{N}(ds,dz).                                                                  
\end{eqnarray}

Let denote by $D(t)$ the right hand side of \eqref{Comp} and set $M(t)=\int_{0}^{t}\beta(s)dW(s) + \int_{0}^{t}\int_{E}\gamma(s)\tilde{N}(ds,dz)$. In view of Girsanov theorem, the process $(D(t))_{0\leq t\leq T}$ is a martingale under the probability measure $\Q$ defined by $\Q=\mathcal{E}_{\sigma}(M).\P$, where $\mathcal{E}_{\sigma}(M)$ is called a Doléan-Dade exponential.  Taking conditional expectation with respect to $\mathcal{F}_{t}$ under $\Q$ both sides of \eqref{Comp}, we obtain $R(t)\Delta X^{\tau}(t)\leq 0\; \Q$-a.s., and hence $\P$-a.s. Finally, since the process $(R(t), t\geq 0)$ is non-negative, we have $t\in[0,\sigma],\; X^{\tau,1}(t)\geq X^{\tau,2}(t)\; \P$-a.s.
\end{proof}
\subsection{Properties of dynamic risk measures } 

\subsection{Optimal stopping problem for dynamic risk measures}
Let $T>0$ be a time horizon and $f$ be delayed generator satisfied ({\bf A2}). For each stopping time $\tau$ with values in $[0,T]$ and $(\psi(t))_{t\geq 0}$ a $(\mathcal{F}_{t})_{t\geq 0}$-adapted square integrable stochastic process, we consider the risk of $\psi(\tau)$ at time $t$ defined by
\begin{eqnarray*}
\rho^{\psi,\tau}(t)=-X^{\psi,\tau}(t),\; 0\leq t\leq \tau,
\end{eqnarray*}
 where $X^{\psi,\tau}$ satisfy BSDE \eqref{BSDEjump} with driver $f{\bf 1}_{[0,\tau]}$, terminal condition $\psi(\tau)$ and terminal time $\tau$. The functional $\rho: (\psi,\tau)\mapsto \rho^{\psi,\tau}(.)$ defines then a dynamic risk measure induced by the BSDE \eqref{BSDEjump} with driver $f{\bf 1}_{[0,\tau]}$. Let us now deal with some optimal stopping problem related to the above risk measure. Contrary to the case without delay, there is a real difficulty in setting up the problem for the BSDE with delayed generator. Indeed, since the comparison principle of delayed BSDEs failed at the neighborhood of $0$, we are no longer able to construct the supremum of this risk on $[0,T]$. To work around this difficulty, we need to construct a stochastic interval in which, we can derive a comparison theorem. For a stopping time $\delta$, let also consider $(X^{\psi,\delta},Z^{\psi,\delta})$ the solution of BSDE \eqref{BSDEjump} with driver $f{\bf 1}_{[0,\delta]}$, terminal condition $\psi(\delta)$ and terminal time $\delta$. We consider following stopping times
 \begin{eqnarray*}
\sigma_n=\inf(A_n)\wedge T,
\end{eqnarray*}
where 
\begin{eqnarray*}
	A_n=\left\{
\begin{array}{ll}
t\geq 0, & \displaystyle \inf_{\tau,\delta}(|X^{\psi,\tau}(t)-X^{\psi,\delta}(t)|\vee|Z^{\psi,\tau}(t)-Z^{\psi,\delta}(t)|\vee \int_{E}|U^{\psi,\tau}(t,z)-U^{\psi,\delta}(t,z)|m(dz) \leq \frac{1}{n}\\ \mbox{or}\\
& \displaystyle\inf_{\tau,\delta}(|X^{\psi,\tau}(t)-X^{\psi, \delta}(t)|\vee |Z^{\psi,\tau}(t)-Z^{\psi,\delta}(t)|\vee \int_{E}|U^{\psi,\tau}(t,z)-U^{\psi,\delta}(t,z)|m(dz))\geq n
\end{array}
\right\}
\end{eqnarray*}
and set
 \begin{eqnarray}
 \overline{\sigma}=sup_{n\geq 1}\sigma_n.\label{ST}
 \end{eqnarray}
 For a stopping time $\sigma\leq \overline{\sigma}$, let consider $\mathcal{F}_{\sigma}$-measurable random variable $v(\sigma)$ (unique for the equality in the almost sure sense) defined by
\begin{eqnarray}\label{inf}
 v(\sigma)=ess\inf_{\sigma\leq \tau\leq T}\rho^{\psi,\tau}(\sigma).
\end{eqnarray}
 Since $\rho^{\psi,\tau}=-X^{\psi,\tau}$, we get
\begin{eqnarray}\label{sup}
 v(\sigma)=ess\inf_{\sigma\leq \tau\leq T}(-X^{\psi,\tau}(\sigma))=-ess\sup_{\sigma\leq \tau\leq T}X^{\psi,\tau}(\sigma),
\end{eqnarray}
for each stopping time $\sigma\in [0,\overline{\sigma}]$, which characterize the minimal risk-measure. We then provide an existence result of an $\sigma$-optimal stopping time $\tau^{*}\in [\sigma, T]$, satisfies $v(\sigma)=\rho^{\psi,\tau^{*}}(\sigma)$ a.s. 

In order to characterize minimal risk measure by reflected BSDEs with jump and delayed generators, let's derive first the notion of solution of this type of equations.
 
\begin{definition}
The triple of processes $(Y(t),Z(t),U(t,z),K(t))_{0 \leq t \leq T,z \in E}$ is said to be a solution of the reflected delayed BSDEs with jumps associated to delayed generator $f$, stochastic terminal times $\tau$, terminal value $\xi$ and obstacle process $(S(t))_{t\geq 0}$, if it satisfies the following.
\begin{enumerate}
\item [(i)] $(Y,Z,U,K)\in \mathcal{S}^{2}(\R)\times\mathcal{H}^{2}(\R)\times\mathcal{H}_{m}^{2}(\R)\times\mathcal{K}^{2}(\R)$.
\item [(ii)] 
\begin{eqnarray}                                                                                             Y(t) &=& \xi + \int_{t}^{\tau}f(s,Y_{s},Z_{s},U_{s}(.))ds + K(\tau)- K(t) -\int_{t}^{\tau}Z(s)dW(s)\nonumber \\                                                                                            && -   \int_{t}^{\tau}\int_{E}U(s,z)\tilde{N}(ds,dz), \;\; 0 \leq t \leq \tau
\label{Eq1} 
\end{eqnarray}
\item [(iii)] $Y$ dominates $S$, i.e. $Y(t) \geq S(t),\;\; 0 \leq t \leq \tau$
\item [(iv)] the Skorohod condition holds:  
$\displaystyle \int_{0}^{\tau}(Y(t^-) - S(t^-))dK(t) = 0$ a.s.
\end{enumerate}
\end{definition}
In our definition, the jumping times of process $Y$ is not come only from Poisson process jumps (inaccessible jumps) but also from the jump of the obstacle process $S$ (predictable jumps).
\begin{remark}
Let us point out that condition $(iv)$ is equivalent to : If $K=K^c+K^d$, where $K^{c}$ and $K^{d}$ denote respectively  continuous and discontinuous part of $K$, then $\displaystyle \int_{0}^{\tau}(Y(t)-S(t))dK^c(t)= 0$ a.s. and for every predictable stopping time $\sigma\in[0,T]$, $\displaystyle \Delta Y(\sigma)=Y(\sigma)-Y(\sigma^-)=-(S(\sigma^-)-Y(\sigma))^{+}{\bf 1}_{[Y(\sigma^-)=S(\sigma^{-})]}$. On the other hand, since the jumping times of the Poisson process are inaccessible, for every predictable stopping time $\sigma\in[0,T]$, \newline $\Delta Y(\sigma)=-\Delta K(\sigma)=-(S(\sigma^-)-Y(\sigma))^{+}{\bf 1}_{[Y(\sigma^-)=S(\sigma^{-})]}$
\end{remark}
The following theorem will be state in special context that $\xi=\psi(\tau)$ and $S=\psi$ in order to establish a link between the risk measure associated with the EDSR $(\tau, \psi (\tau), f)$ and the solution of the reflected EDSR associated with $(\tau, \psi (\tau), f,\psi)$.

\begin{theorem}\label{Theo3.1}
Let $\tau$ be a stopping time belonging on $[0,T],\,\{\psi(t), \, 0\leq t\leq T\}$ and $f$ be respectively a terminal time, an rcll process in $\mathcal{S}^{2}(\R)$ and a delayed generator satisfying Assumption $({\bf A3})-({\bf A4})$. Suppose $(Y,Z,U,K)$ be the solution of the reflected BSDE associated to $(\tau,\psi(\tau), f, \psi)$.
\begin{enumerate}
\item [(i)] For each stopping time $\sigma\leq \overline{\sigma}$, we have 
\begin{eqnarray}\label{mini}
v(\sigma)=-Y(\sigma) = -ess \sup_{\tau \in [\sigma,T]}X^{\psi,\tau}(\sigma),  
\end{eqnarray}
where $v(\sigma)$ is defined by \eqref{inf}.
\item [(ii)] For each stopping time $\sigma$ with values on $[0,\overline{\sigma}]$ and each $\varepsilon>0$, let $D^{\varepsilon}_{\sigma}$ be the stopping time defined by
\begin{eqnarray}
D^{\varepsilon}_{\sigma}=\inf\left\{t\in[\sigma,T],\; Y(t)\leq \psi(t)+\varepsilon\right\}.\label{TA}
\end{eqnarray}
We have 
\begin{eqnarray*}
Y(\sigma)\leq X^{\psi,D^{\varepsilon}_{\sigma}}(\sigma)+C\varepsilon\;\; \mbox{a.s.},	
\end{eqnarray*} 
\end{enumerate}  
where $C$ is a constant which only depends on $T$ and the Lipschitz constant $K$. In other words, $D^{\varepsilon}_{\sigma}$ is a $(C\varepsilon)$-optimal stopping time for \eqref{mini}.
\end{theorem}
\begin{remark}
Note that Property $(ii)$ implies that for all stopping times $\sigma$ and $\tau$ with values on $[0,\overline{\sigma}]$ and $[0,T]$ respectively such that $\sigma \leq \tau\leq D^{\varepsilon}_{\sigma}$, we have $Y(\sigma)=\mathcal{E}^{f}_{\sigma,\tau}(Y(\tau))$ a.s. In other words, the process $(Y(t),\;  \sigma \leq  t \leq D^{\varepsilon}_{\sigma})$ is an $\mathcal{E}^{f}$-martingale.
\end{remark}

\begin{proof}[Proof of Theorem \ref{Theo3.1}]
Let consider $\sigma$ and $\tau$ two stopping time with values in $[0,T]$ such that $\sigma\leq \tau$. Let consider $(Y,Z,U,K)$ be solution of the reflected BSDE associated to $(\psi(\tau), f, \psi)$. We have 
\begin{eqnarray*}                                                                                         Y(\sigma)& = & \psi(\tau) + \int_{\sigma}^{\tau}f(s,Y_{s},Z_{s},U_{s}(.))ds + K(\tau)- K(\sigma) - \int_{\sigma}^{\tau}Z(s)dW(s)\\                                                                          &&-\int_{\sigma}^{\tau}\int_{E}U(s,z)\tilde{N}(ds,dz)
\end{eqnarray*}
According to reflected BSDEs framework, we know that the process $K$ is non-decreasing, hence $K(\tau)-K(\sigma)\geq 0$. Therefore,
\begin{eqnarray}\label{compari}
Y(\sigma)& \geq & \psi(\tau) + \int_{\sigma}^{\tau}f(s,Y_{s},Z_{s},U_{s}(.))ds  - \int_{\sigma}^{\tau}Z(s)dW(s)                                                                                                    - \int_{\sigma}^{\tau}\int_{E}U(s,z)\tilde{N}(ds,dz).
\end{eqnarray}
Let $(\bar{Y},\bar{Z},\bar{U})$ satisfy equation
\begin{eqnarray}
\bar{Y}(\sigma)&=&\psi(\tau) + \int_{\sigma}^{\tau}f(s,\bar{Y}_{s},\bar{Z}_{s},\bar{U}_{s}(.))ds  - \int_{\sigma}^{\tau}\bar{Z}(s)dW(s)                                                                                                    - \int_{\sigma}^{\tau}\int_{E}\bar{U}(s,z)\tilde{N}(ds,dz).	
\end{eqnarray}
It follows from \eqref{compari} that $Y(\sigma)\geq\bar{Y}(\sigma)$. On other hand, thanks to uniqueness of solution for BSDE \eqref{BSDEjump}, we obtain $\bar{Y}=X^{\psi,\tau}$ which implies $Y(\sigma)\geq X^{\psi,\tau}(\sigma)$ for all $\tau\in[\sigma, T]$. Finally we get 
  
\begin{eqnarray}\label{compari1}
Y(\sigma)&\geq & ess\sup_{\tau\in[\sigma,T]}X^{\psi,\tau}(\sigma).
\end{eqnarray}
Let us show now the reversed inequality. In view of it definition, $D^{\varepsilon}_{\sigma}$ belongs in $[\sigma,T]$ and for each $t\in[\sigma(\omega),D_{\sigma}(\omega)[$ for almost all $\omega\in \Omega$, we have $Y(t)>\psi(t)$ a.s. Therefore, recalling reflected BSDEs framework, the function $t\mapsto K(t)$ is almost surely constant on $[\sigma(\omega),D_{\sigma}(\omega)]$ so that $K(D_{\sigma})-K(\sigma)=0$. This implies that
\begin{eqnarray*}
Y(\sigma)& = & \psi(D_{\sigma}) + \int_{\sigma}^{D_{\sigma}}f(s,Y_{s},Z_{s},U_{s}(.))ds  - \int_{\sigma}^{D_{\sigma}}Z(s)dW(s)-\int_{\sigma}^{D_{\sigma}}\int_{E}U(s,z)\tilde{N}(ds,dz).
\end{eqnarray*}
Using again comparison principle, we derive that $Y(\sigma)=X^{\psi,D_{\sigma}}(\sigma)$ which leads
\begin{eqnarray}\label{compari2}
Y(\sigma)&\leq & ess \sup_{\tau \in [\sigma, T]}X^{\psi,\tau}(\sigma)
\end{eqnarray}
According to \eqref{compari1} and \eqref{compari2}, we prove $(i)$. We will prove now $(ii)$. According to \eqref{TA} and comparison theorem of BSDE with delayed generator, we get that for all stopping times $\sigma\leq \bar{\sigma}$,
\begin{eqnarray}
Y(\sigma)=X^{Y,D_{\sigma}^{\varepsilon}}(\sigma)\leq X^{\psi+\varepsilon,D^{\varepsilon}_{\sigma}}(\sigma)\;\;\;\;\;\mbox{as}.\label{ZA}
\end{eqnarray}
On the other hand, using some appropriate estimate on BSDE with delayed generator, we derive
\begin{eqnarray*}
|X^{Y,D_{\sigma}^{\varepsilon}}(\sigma)-X^{\psi+\varepsilon,D^{\varepsilon}_{\sigma}}(\sigma)|^2\leq e^{\beta(T-S)}\varepsilon^2, \;\;\;\;\;\mbox{as},
\end{eqnarray*}
where $\beta$ is a constant depending only on the time horizon $T$ and a Lipschitz constant $K$. Finally, in view of \eqref{ZA} we get the result.
\end{proof}
To end this subsection let now derive an optimality criterium for the optimal stopping time problem based on the strict comparison theorem. Before let us give what we mean by an optimal stopping time. 
\begin{definition}
A stopping time $\bar{\tau}\in [\sigma,T]$ is an $\sigma$-optimal stopping time if  
\begin{eqnarray*}
Y(\sigma) = ess \sup_{\tau \in [\sigma,T]}X^{\psi,\tau}(\sigma)= X^{\psi, \bar{\tau}}(\sigma). 
\end{eqnarray*}
On the other word, the process $(Y(t))_{\sigma\leq t \leq \bar{\tau}}$ is the solution of the non reflected BSDE associated with terminal time $\bar{\tau}$ and terminal value $\psi(\bar{\tau})$.
\end{definition}
\begin{theorem}
Let a rcll process $(\psi(t))_{t\geq 0}$ be l.u.s.c along stopping times and belong to $\mathcal{S}^{2}(\R)$. We assume  $({\bf A1})$-$({\bf A4})$ holds and suppose $(Y,Z,U(.),K)$ is a solution of the reflected BSDE with jump and delayed \eqref{Eq1}. Setting for all stopping time $\sigma\leq \overline{\sigma}$ ($\overline{\sigma}$ is the same defined by \eqref{ST}), the following stopping times:
\begin{eqnarray}\label{ST1}
	\tilde{\tau}_{\sigma} = \lim_{\epsilon \downarrow 0}\uparrow \tau_{\sigma}^{\epsilon},
\end{eqnarray}
where $\tau_{\sigma}^{\varepsilon}=\inf\{\sigma\leq t\leq T,\;\; Y(t)\leq \psi(t)+\varepsilon\}$,
\begin{eqnarray}\label{ST2}
\tau_{\sigma}^{*}=\inf\{\sigma \leq t \leq T,\; Y(t) = \psi(t) \},
\end{eqnarray}
and 
\begin{eqnarray}\label{ST3}
\widetilde{\tau}_{\sigma}=\inf \left\{\sigma \leq t \leq T,\; K(t)- K(\sigma) > 0 \right\}.	
\end{eqnarray} 
Then $\overline{\tau}_{\sigma},\, \tau_{\sigma}^{*}$ and $\widetilde{\tau}_{\sigma}$ are $\sigma$-stopping times of the optimal problem \eqref{inf} such that
\begin{itemize}
\item[(i)] $\overline{\tau}_{\sigma}\leq \tau^{*}_{\sigma}$ and we have $Y(s)=X^{\psi,\tau_{\sigma}^{*}}(s)$ for all $\sigma\leq s\leq \tau_{\sigma}^{*}$\;\; a.s. 
\item [(ii)] $\overline{\tau}_{\sigma}$ is the minimal $\sigma$-stopping time
\item [(iii)] $\widetilde{\tau}_{\sigma}$ is the maximal $\sigma$-stopping time.
\item [(iv)] Moreover if in $({\bf A3})(iv)$, we have $|\theta^{x_t,z_t,u_t(.),u'_t(.)}|>-1$, then $\tau^{*}_{\sigma}=\overline{\tau}_{\sigma}$.
\end{itemize} 
\end{theorem}
Since the proof follows the same argument used in its  proof and to avoid unnecessarily lengthening the writing, we will refer the reader to the proof of Theorem 3.7 appear to \cite{10}.

\section{Reflected BSDEs with jumps and time-delayed generator}
This section is devoted to study in general framework of the reflected BSDEs with jumps, right continuous and left limit (rcll) obstacle process and delayed generator.  More precisely, for a fixed $T>0$ and a stopping time $\tau$ in value on $[0,T]$, we consider
\begin{eqnarray}
Y(t)&=&\xi+ \int_{t}^{\tau}f(s,Y_{s},Z_{s},U_{s}(.))ds + K(\tau)- K(t) -\int_{t}^{\tau}Z(s)dW(s)\nonumber \\                                                                                            && -   \int_{t}^{\tau}\int_{E}U(s,z)\tilde{N}(ds,dz), \;\; 0 \leq t \leq \tau.
\label{RBSDE} 
\end{eqnarray}
We derive an existence and uniqueness result under the following additional hypothesis related to the obstacle process.
 \begin{description} 
\item [$({\bf A5})$] The obstacle process $\{S(t),\;\; 0 \leq t \leq T \}$ is a rcll progressively measurable $\R$-valued process satisfies
\begin{itemize}
\item [$(i)$] $\E \left( \sup_{0 \leq t \leq T}(S^{+}(t))^{2}\right) < +\infty$,
\item [$(ii)$] $\xi\geq S(\tau)$ a.s.
\end{itemize}
\end{description}
To begin with, let us first assume $f$ to be independent of $(y_t, z_t, u_t)\in $, that is, it is a given $(\mathcal{F}_t)_{0\leq t\leq \tau}$-progressively measurable process satisfying that $\E\left(\int^{\tau}_{0}f(t)dt\right)<+\infty$. A solution to the backward reflection problem (BRP, in short) is a triple $(Y , Z , U, K )$ which satisfies $(i), (iii), (iv)$ of the Definition 2.4. and	
\begin{itemize}
\item [(ii')] 
\begin{eqnarray*}
	Y(t)&=&\xi+ \int_{t}^{\tau}f(s)ds + K(\tau)- K(t) -\int_{t}^{\tau}Z(s)dW(s)-\int_{t}^{\tau}\int_{E}U(s,z)\tilde{N}(ds,dz), \;\; 0 \leq t \leq \tau.
\end{eqnarray*}
\end{itemize}
The following proposition is from Hamadène and Ouknine  \cite{E5} (Theorem 1.2.a and 1.4.a) or Essaky \cite{Essaky}.
\begin{proposition}
The reflected BSDE with jump associated with $(\xi, g, S)$ has a unique solution $(Y,Z,K,U)$.
\end{proposition}
 
\begin{theorem} \label{Theo 2.2}
Assume $({\bf A1})$-$({\bf A3})$ and $({\bf A5})$ hold. For a sufficiently small time horizon $T$ or for a sufficiently small Lipschitz constant $K$ of the generator $f$ i.e 
\begin{eqnarray}
KTe\max\{1,T\}< 1, \label{C1}
\end{eqnarray}
the reflected BSDE with jumps and delayed generator \eqref{Eq1} admits a unique solution $(Y,Z,U,K)\in\mathcal{S}^{2}(\R)\times\mathcal{H}^{2}(\R)\times\mathcal{H}_{m}^{2}(\R)\times \mathcal{K}^{2}(\R)$.   
\end{theorem}

\begin{proof}
Let us begin with the uniqueness result. In this fact, assume $(Y,Z,U,K)$ and $(Y',Z',U',K')$ be two solutions of RBSDE associated to data $(\xi,f,S)$ and set $\overline{\theta}=\theta-\theta'$ for $\theta=Y, Z, U,K$. Applying Itô's formula to the discontinuous semi-martingale $|\overline{Y}|^2$, we have
\begin{eqnarray}\label{i1}
&&|\overline{Y}(t)|^2+\int_t^{T}|\overline{Z}(s)|^2ds+\int_{t}^{T}\int_{\mathcal{E}}|\bar{U}(s,z)|^2 m(dz)ds\nonumber\\
&=&2\int_t^T\overline{Y}(s)(f(s,Y_s,Z_s,U_s(.))-f(s,Y'_s,Z'_s,U'_s(.)))ds+2\int_t^T\overline{Y}(s)d\overline{K}(s)\nonumber\\
&&-2\int_t^T\overline{Y}(s)\overline{Z}(s)dW(s)-2\int_t^T\int_{\mathcal{E}}\overline{Y}(s^-)\overline{U}(s,z)\tilde{N}(ds,dz).
\end{eqnarray}
In view of Skorohod condition $(iv)$, we get
\begin{eqnarray}\label{i2}
\int_t^T\overline{Y}(s)d\overline{K}(s)&=&\int_t^T(Y(s^-)-S(t^-))dK(t)+\int_t^T(S(s^-)-Y'(t^-))dK(t)\nonumber\\
&&+ \int_t^T(Y'(s^-)-S'(t^-))dK'(t)+\int_t^T(S'(s^-)-Y(t^-))dK'(t)\nonumber\\
&\leq & 0.
\end{eqnarray}
Next, since the third and fourth term of \eqref{i1} are $(\mathcal{F}_t)_{t\geq 0}$-martingales together with \eqref{i2}, we have
\begin{eqnarray}\label{i5}
&&\E\left(|\overline{Y}(t)|^2+\int_t^{T}|\overline{Z}(s)|^2ds+\int_{t}^{T}\int_{\mathcal{E}}|\bar{U}(s,z)|^2m(dz)ds\right)\nonumber\\
&=&2\E\left(\int_t^T\overline{Y}(s)(f(s,Y_s,Z_s,U_s(.))-f(s,Y'_s,Z'_s,U'_s(.)))ds\right)\nonumber\\
&\leq & \beta\E\left( \int_t^T|\overline{Y}(s)|^2ds\right)+\frac{1}{\beta}\E\left(\int_t^T|f(s,Y_s,Z_s,U_s(.))-f(s,Y'_s,Z'_s,U'_s(.)))|^2ds\right)
\end{eqnarray}
According to assumptions $({A3})(i)$, change of variable and fubini's theorem, we obtain
\begin{eqnarray}\label{i6}
&& \int_t^T|f(s,Y_s,Z_s,U_s(.))-f(s,Y'_s,Z'_s,U'_s(.)))|^2ds \nonumber\\
&\leq & K\int_t^T\left(\int_{-T}^0\left[|\overline{Y}(s+u)|^{2}+|\overline{Z}(s+u)|^{2}+\int_{\mathcal{E}}|U(s+u,z)|^2 m(dz)\right]\alpha(du)\right)ds\nonumber\\
&\leq & K\int_{-T}^{T}\left[|\overline{Y}(s)|^{2}+|\overline{Z}(s)|^{2}+\int_{\mathcal{E}}|U(s,z)|^2 m(dz)\right]ds.
\end{eqnarray}
Putting the last inequality into \eqref{i5} yields
\begin{eqnarray}\label{i6}
&&\E\left(|\overline{Y}(t)|^2+\int_t^{T}|\overline{Z}(s)|^2ds+\int_{t}^{T}\int_{\mathcal{E}}|\bar{U}(s,z)|^2 m(dz)ds\right)\nonumber\\
&\leq & \left(\beta+\frac{K}{\beta}\right)\E\int_{-T}^T|\overline{Y}(s)|^2ds+\frac{K}{\beta}\E\int_0^T\left(|\overline{Z}(s)|^{2}+\int_{\mathcal{E}}|U(s,z)|^2 m(dz)\right)ds.
\end{eqnarray}
If we choose $\beta$ such that $\frac{K}{\beta}\leq 1$, inequality $\eqref{i6}$ becomes
\begin{eqnarray}\label{i7}
&&\E\left(|\overline{Y}(t)|^2+\int_0^{T}|\overline{Z}(s)|^2ds+\int_{0}^{T}\int_{\mathcal{E}}|\bar{U}(s,z)|^2 m(dz)ds\right)\nonumber\\
&\leq & C\E\int_{-T}^T|\overline{Y}(s)|^2ds.
\end{eqnarray}
According the above estimate, using Gronwall's lemma and in view of the right continuity of the process $\overline{Y}$, we have $Y=Y'$. Therefore $(Y,Z,U,K) = (Y',Z',U',K')$, whence reflected BSDE with jump and delayed generator \eqref{RBSDE} admit a uniqueness solution.

It remains to show the existence which will be obtained via a fixed point method. For this let consider $\mathcal{D}=\mathcal{S}^{2}(\R)\times \mathcal{H}^{2}(\R)\times \mathcal{H}_m(\R)$ endowed with the norm $\|(Y,Z,U)\|_{\beta}$ defined by
\begin{eqnarray*}
\|(Y,Z,U)\|_{\beta}=\E\left(\sup_{0\leq t\leq \tau}e^{\beta t}|Y(t)|^2+\int_0^{\tau}e^{\beta t}\left(|Z(t)|^2+\int_{\mathcal{E}}U(s,z)m(dz)\right)ds\right).
\end{eqnarray*}
We now consider a mapping $\Phi:\mathcal{D}$ into itself defined by $\Phi((Y,Z,U))=(\tilde{Y},\tilde{Z},\tilde{U})$ which means that there is a process $\tilde{K}$ such as  $(\tilde{Y},\tilde{Z},\tilde{U},\tilde{K})$ solve the reflected BSDE with jump associated to the data $\xi, f(t,Y,Z,U)$ and $S$. More precisely, $(\tilde{Y},\tilde{Z},\tilde{U},\tilde{K})$ satisfies $(i),\,(iii),\, (iv)$ of Definition 2.4 such that
\begin{eqnarray*}
\tilde{Y}(t)=\xi+\int_t^{\tau}f(s,Y_s,Z_s,U_s(.))ds+\tilde{K}(\tau)-\tilde{K}(t)-\int^{\tau}_t\tilde{Z}(s)dW(s)-\int_t^{\tau}\int_{E}\tilde{U}(s,z)\tilde{N}(ds,dz).
\end{eqnarray*}
For another process $(Y',Z',U')$ belonging in $\mathcal{D}$ let set $\Phi(Y',Z',U')=(\tilde{Y}',\tilde{Z}',\tilde{U}')$. In the sequel and for a generic process $\theta$, we denote $\delta\theta=\theta-\theta'$. Next, applying Ito's formula to $e^{\beta t}|\Delta\tilde{Y}(t)|^{2}$ yields
\begin{eqnarray*}
&&e^{\beta t}|\delta\tilde{Y}(t)|^2+ \beta\int_{t}^{T}e^{\beta s}|\delta \tilde{Y}(s)|^2ds+\int_{t}^{T}e^{\beta s}|\delta \tilde{Z}(s)|^2ds\\
&&+ \int_{t}^{T}e^{\beta s}\int_{\mathcal{E}}|\delta \tilde{U}(s,z)|^2m(dz)ds +\sum_{t\leq s\leq T}e^{\beta s}(\Delta_s(\delta \tilde{Y})-\Delta_s(\delta \tilde{Y}'))^2\\
&=&  2\int_{t}^{T}e^{\beta s}\delta\tilde{Y}(s)(f(s,Y_s,Z_s,U_s(.))-f(s,Y'_s,Z'_s,U_s(.)))ds+2\int_{t}^{T}e^{\beta s}\delta\tilde{Y}(s)d\delta\tilde{K}(s)\\
 &&+M(T)-M(t),
\end{eqnarray*}
where $(M(t))_{0\leq t\leq T}$ is a martingale. On the other hand, in view of uniqueness proof and young inequality, we have respectively $\displaystyle \int_{t}^{T}e^{\beta s}\delta\tilde{Y}(s)d\delta\tilde{K}(s)\leq 0$ and
\begin{eqnarray*}
&& e^{\beta s}\delta\tilde{Y}(s)(f(s,Y_s,Z_s,U_s(.))-f(s,Y'_s,Z'_s,U_s(.)))ds\\
&\leq & \beta	e^{\beta t}|\delta\bar{Y}(s)|^2+\frac{1}{\beta}|f(s,Y_s,Z_s,U_s(.))-f(s,Y'_s,Z'_s,U_s(.))|^2,
\end{eqnarray*}
which allow us to get
\begin{eqnarray}\label{J1}
&&e^{\beta t}|\delta\tilde{Y}(t)|^2+\int_{t}^{T}e^{\beta s}|\delta \tilde{Z}(s)|^2ds+\int_{t}^{T}e^{\beta s}\int_{\mathcal{E}}|\delta \tilde{U}(s,z)|^2m(dz)ds +\sum_{t\leq s\leq T}e^{\beta s}(\Delta_s(\delta \tilde{Y})-\Delta_s(\delta \tilde{Y}'))^2\nonumber\\
&\leq &
\frac{1}{\beta}\int_{0}^{T}e^{\beta s}|f(s,Y_s,Z_s,U_s(.))-f(s,Y'_s,Z'_s,U_s(.))|^2ds+M(T)-M(t).
\end{eqnarray}
Then taking the conditional expectation with respect $(\mathcal{F}_t)_{t\geq 0}$ in both side of the previous inequality, we obtain  
\begin{eqnarray*}\label{j2}
e^{\beta t}|\delta\bar{Y}(t)|^2 
&\leq & \frac{1}{\beta}\E\left(\int_{0}^{T}e^{\beta s}|f(s,Y_s,Z_s,U_s(.))-f(s,Y'_s,Z'_s,U_s(.))|^2 ds|\mathcal{F}_t\right), 
\end{eqnarray*}
which together with Doob inequality yields
\begin{eqnarray}\label{J2}
\E\left(\sup_{0\leq t\leq T}e^{\beta t}|\delta\bar{Y}(t)|^2 \right)
&\leq & \frac{1}{\beta}\E\left(\int_{0}^{T}e^{\beta s}|f(s,Y_s,Z_s,U_s(.))-f(s,Y'_s,Z'_s,U_s(.))|^2 ds\right). 
\end{eqnarray}
Taking expectation in both side of \eqref{J1} for $t=0$, it follows from \eqref{J2} that
\begin{eqnarray}\label{J3}
&&\E\left(\sup_{0\leq t\leq T}e^{\beta t}|\delta\bar{Y}(t)|^2+ \int_{t}^{T}e^{\beta s}|\delta \tilde{Z}(s)|^2ds+\int_{t}^{T}e^{\beta s}\int_{\mathcal{E}}|\delta \tilde{U}(s,z)|^2m(dz)ds\right)\nonumber\\
&\leq & \frac{1}{\beta}\E\left(\int_{0}^{T}e^{\beta s}|f(s,Y_s,Z_s,U_s(.))-f(s,Y'_s,Z'_s,U_s(.))|^2 ds\right). 
\end{eqnarray}
Let us now derive the estimation of right side of inequality \eqref{J3}. In view of assumption $({\bf A1})$, we have
\begin{eqnarray*}
&&\int_{0}^{T}e^{\beta s}|f(s,Y_s,Z_s,U_s(.))-f(s,Y'_s,Z'_s,U_s(.))|^2 ds\\
&\leq& K\int_0^T\int_{-T}^{0}e^{\beta s}\left(|\delta Y(s+u)|^2+|\delta Z(s+u)|^2+\int_{\mathcal{E}}|\delta U(s+u,z)|^2m(dz)\right)\alpha(du)ds.
\end{eqnarray*}
Next, since $Z(t) = 0, U(t,.) \equiv 0$ and  $Y(t)=Y(0)$, for $t<0$, we get respectively with Fubini's theorem, changing the variables that 
\begin{eqnarray}\label{J4}
&&\int_{0}^{T}e^{\beta s}|f(s,Y_s,Z_s,U_s(.))-f(s,Y'_s,Z'_s,U_s(.))|^2 ds\nonumber\\
&\leq & K\max(1,T)e^{\beta T}\left(\sup_{0\leq t\leq T}e^{\beta t}|\delta Y(t)|^2+\int_0^T e^{\beta s}\left(|\delta Z(s)|^2+\int_{\mathcal{E}}|\delta U(s,z)|^2m(dz)\right)ds\right).\nonumber\\
\end{eqnarray}
Thereafter, it follows from \eqref{J3}, \eqref{J4} and $\beta=\frac{1}{T}$ that 
\begin{eqnarray*}
&&\E\left[\sup \limits_{0 \leq t \leq T} e^{\beta t}\vert \delta\tilde{Y}(t)\vert^{2}+\int_{0}^{T}e^{\beta t}\vert \delta\tilde{Z}(t)\vert^{2}dt+\int_{0}^{T}\int_{\mathcal{E}} e^{\beta t}\vert \delta\tilde{U}(t, z)\vert^{2}m(dz)dt \right]\\
&\leq &KTe\max(1,T)\E\left(\sup_{0\leq t\leq T}e^{\beta t}|\delta Y(t)|^2+\int_0^T e^{\beta s}\left(|\delta Z(s)|^2+\int_{\mathcal{E}}|\delta U(s,z)|^2m(dz)\right)ds\right),
\end{eqnarray*}
which mean that 
\begin{eqnarray*}
\|\Phi(Y,Z,U)-\Phi(Y',Z',U')\|_{\beta}\leq KTe\max(1,T)\|(\delta Y,\delta Z,\delta U)\|_{\beta}.
\end{eqnarray*}
For a sufficiently small $T$ or $K$, i.e, $KTe\max(1,T)<1$, the function $\Phi$ is a contraction. Consequently $\Phi$ admits a unique fixed point $(Y,Z,U)$ i.e $(Y,Z,U)=\Phi((Y,Z,U)$ and there is a nondecreasing process $K$ such that $(Y,Z,U,K)$ is solution of the RBSDE \eqref{Eq1}.
\end{proof}
\section{Comparison principle for reflected BSDEs with jumps and delayed generator and optimization problem}
\subsection{Comparison principle for reflected BSDEs with jumps and delayed generator}
In this subsection we give a comparison principle to the reflected BSDEs with jumps and delayed generator. The proof is simple and based on the characterization of solutions of reflected BSDEs with jumps and delayed generator established in Theorem \ref{Theo 2.2} and the comparison theorem for non reflected BSDEs with jumps and delayed generator. Therefore, unlike without delay, result is valid only in a random interval $[0, \bar{\sigma}]$, with $\overline{\sigma}$ defined by \eqref {ST}. Let $(Y^i,Z^i,U^i,K^i)$ is a unique solution of reflected BSDE with jump and delayed generator associated to $(\tau,\psi^i,f^i),\;\; i=1,2$. 
\begin{theorem}\label{CP}
Let $\psi,\; \psi'$ and $f^1$, $f^2$ be respectively two rcll obstacle processes and two Lipschitz drivers satisfying ({\bf A3})-({\bf A5}). Suppose
\begin{itemize}
\item [(i)] $\psi^1(t)\leq \psi^2(t)$, a.s. for all $t\in[0,\tau]$
\item [(ii)] for all $t\in [0,\tau],\; f^1(t,Y^1_t,Z^1_t,U^1_t)\leq f^2(t,Y^1_t,Z^1_t,U^1_t)$, a.s. or $f^1(t,Y^2_t,Z^2_t,U^2_t)\leq f^2(t,Y^2_t,Z^2_t,U^2_t)$.
\end{itemize}
Then there exists a stopping times $\overline{\sigma}$ (defined by \eqref{ST}) such that 
\begin{eqnarray*}
	Y^1(t)\leq Y^2(t), \; a.s.\;\; t\in [0,\overline{\sigma}].
\end{eqnarray*}
\end{theorem}

\begin{proof}
Let denote by $X^{\psi^i,\tau}$ the unique solution of BSDE associated with $(\tau,\psi^i, f^i)$ for $i=1, 2$. In view of Theorem 2.3, we have
\begin{eqnarray*}
X^{\psi^1,\tau}(t)\leq X^{\psi^2,\tau}(t),\;\;\; \mbox{a.s}.,
\end{eqnarray*}
for a fix $t\in [0,\overline{\sigma}\wedge \tau]$. 
Next, taking the essential supremum over all stopping times $\tau$ to values in $[t,T]$, it follows from Theorem \ref{Theo3.1} that
\begin{eqnarray*}
Y^1(t)=\sup_{\tau\in [t,T]} X^{\psi^1,\tau}(t)\leq \sup_{\tau\in [t,T]}X^{\psi^2,\tau}(t)=Y^2(t),\;\;\;\; \mbox{a.s}.
\end{eqnarray*}
\end{proof}
\subsection{Optimization problem for reflected BSDEs with jump and delayed generator}
This subsection is devoted to establish an optimization problem with the help of above comparison principle. For $\mathcal{A}$ a subset of $\R$, let $\{f^{\delta},\;\; \delta \in \mathcal{A}\}$ be a family of $\R$-valued function defined on $\Omega\times [0,T]\times L^{\infty}_{T}(\R)\times L^{2}_{T}(\R)\times L^{2}_{T,m}(\R)$. We consider $(Y^{\delta},Z^{\delta},U^{\delta}(.))$ a family of solution of reflected BSDEs associated to $(\psi,f^{\delta})$. For a appropriated stopping time $\sigma$ belong in $[0,T]$, let solve the following optimisation problem:
\begin{eqnarray}
v(\sigma)=ess\inf \limits_{\delta \in \mathcal{A}}Y^{\delta}(\sigma).	\label{OP}
\end{eqnarray}
For all $(t,y,z,k)\in[0,T]\times L^{\infty}_{T}(\R)\times L^{2}_{T}(\R)\times L^{2}_{T,m}(\R)$, let set 
\begin{eqnarray*}
f(t,y,z,k)=ess\inf_{\delta\in \mathcal{A}}f^{\delta}(t,y,z,k),\;\;\;\;\;\P\,\mbox{-a.s}.
\end{eqnarray*}

Optimisation problem \eqref{OP} will be treat in two context.

First, we suppose that $f$ is one of generators indexed by $\delta\in\mathcal{A}$, i.e there exists $\overline{\delta}\in\mathcal{A}$ such that for all $(t,y,z,k)\in[0,T]\times L^{\infty}_{T}(\R)\times L^{2}_{T}(\R)\times L^{2}_{T,m}(\R)$,
\begin{eqnarray}\label{mini}
f(t,y,z,k)=f^{\overline{\delta}}(t,y,z,k)\;\; \P\;\mbox{- a.s}.
\end{eqnarray}
Next, we suppose that $f$ does not belong to the above family.

We derive the following two results.
\begin{proposition}\label{Popt}
Assume $({\bf A1})$-$({\bf A4})$ and \eqref{mini}.
 Then, there exists a stopping time $\widehat{\sigma}$ defined by
 \begin{eqnarray}\label{Compopti}
 \widehat{\sigma}=ess\inf_{\delta\in \mathcal{A}}\sigma_{\delta},	
 \end{eqnarray}
where $\sigma_{\delta}$ is defined as in \eqref{ST}, such that for $\sigma\leq \widehat{\sigma}$
 \begin{eqnarray*}
 Y(\sigma) = ess\inf_{\delta \in \mathcal{A}}Y^{\delta}(\sigma)\;\;\;\; a.s,
 \end{eqnarray*}
 where $(Y,Z,U(.))$ is the unique solution of the reflected BSDE associated to $(\psi,f)$.
\end{proposition}
\begin{proof}  
For each $\delta \in \mathcal{A}\;$ and each $\tau\in [\widehat{\sigma},T]$, the comparison theorem for delays BSDEs with jumps yields that for a stopping time $\sigma\leq\widehat{\sigma},\; X^{\psi,\tau}(\sigma) \leq  X^{\delta,\psi,\tau}(\sigma)$. The essential supremum taken over $\tau $ on the both side of the above inequality, we get 
\begin{eqnarray}
ess \sup_{\sigma\leq \tau\leq T}X^{\psi,\tau}(\sigma)\leq ess \sup_{\sigma\leq \tau \leq T}X^{\delta,\psi,\tau}(\sigma).\label{Sup}
\end{eqnarray}
According to the representation \eqref{mini}, it follows from \eqref{Sup} that $Y(\sigma) \leq Y^{\alpha}(\sigma)$ for each $\alpha \in \mathcal{A}$ and each $\sigma\in [0,\widehat{\sigma}]$. This implies by the essential infimum taken over $\delta$ in both side of the previous inequality that for all stopping time $\sigma\in [0,\widehat{\sigma}]$ that
\begin{eqnarray}
	Y(\sigma)\leq ess \inf_{\delta\in \mathcal{A}} Y^{\delta}(\sigma).\label{IZ1}
\end{eqnarray}
On the other, since there exists $\overline{\delta}\in\mathcal{A}$ such that $f=f^{\overline{\delta}}$, in view of uniqueness of reflected BSDE associated to $(f,\psi,\tau)$, we obtain $Y=Y^{\overline{\delta}}$. Therefore 
\begin{eqnarray*}
Y(\sigma) \geq ess \inf_{\delta \in \mathcal{A}} Y^{\delta}(\sigma)
\end{eqnarray*}
which together with \eqref{IZ1} yields for all stopping time $\sigma\leq \widehat{\sigma}$,
\begin{eqnarray*}
Y(\sigma) = ess \inf_{\delta \in \mathcal{A}} Y^{\delta}(\sigma).
\end{eqnarray*}
\end{proof}
\begin{proposition}\label{Poptbis}
Assume $({\bf A1})$-$({\bf A4})$ and suppose that $f\notin\{f^{\delta}, \delta\in \mathcal{A}\}$.
 Then, there exists a stopping time $\widehat{\sigma}$ defined as in Proposition \ref{Popt} such that for $\sigma\leq \widehat{\sigma}$
 \begin{eqnarray*}
 Y(\sigma) = ess\inf_{\delta \in \mathcal{A}}Y^{\delta}(\sigma)\;\;\;\; a.s,
 \end{eqnarray*}
 where $(Y,Z,U(.))$ is the unique solution of the reflected BSDE associated to $(\psi,f)$.
\end{proposition}
\begin{proof}
With the same argument as in previous proof, we obtain for all $\sigma\leq \widehat{\sigma}$, 
\begin{eqnarray}
Y(\sigma)\leq ess \inf_{\delta\in \mathcal{A}} Y^{\delta}(\sigma) \;\;\; a.s.\label{IZ2}
\end{eqnarray}
Let us derive the reversed inequality. According to the definition of $f$, we have the following: $\P$-a.s., for all $\eta>0$, there exists $\delta^{\eta}$ such that $f^{\delta^{\eta}}-\eta\leq f<f^{\delta^{\eta}}$. Moreover, applied Lemma 2.1 appear in \cite{D1} to BSDE with jump and delayed generator, we provide that for all stopping time $\sigma\leq \widehat{\sigma}$, there exists constant $C$ depending only to Lipschitz constant and terminal horizon such that    
\begin{eqnarray*}
X^{\psi, \tau}(\sigma) + C\eta \geq X^{\delta^{\eta},\psi,\tau}(\sigma), \;\;\; a.s.
\end{eqnarray*}
Using argument like in the proof above, we obtain
\begin{eqnarray*} \label{C2}
Y(\sigma) + C\eta \geq ess \inf_{\delta \in \mathcal{A}}Y^{\delta}.(\sigma)
\end{eqnarray*}
Since the inequality holds for each $\eta> 0$ then we have:
\begin{eqnarray*}
Y(\sigma) \geq ess \inf_{\delta \in \mathcal{A}}Y^{\delta}(\sigma),
\end{eqnarray*}
which together with \eqref{IZ2} ends the proof.
\end{proof}
\begin{remark}
According Propositions \ref{Popt} and \ref{Poptbis}, we  establish that the value function of optimisation problem \eqref{OP} associated to a family of functions $\{f^{\delta},\; \delta\in \mathcal{A}\}$ is $Y$ the solution of reflected BSDE with jump and delayed generator $f$ defined by $f=ess\inf_{\delta\in\mathcal{A}}f^{\delta}$.
\end{remark}
\section{Robust optimal stopping problem for delayed risk measure}
In this section, we consider the  ambiguous risk-measures modeling by a BSDE with jump that we do not enough concerning the the delayed generator associated. More precisely, we consider $(\rho^{\delta})_{\delta \in \mathcal{A}}$ the family of the risk-measure of the position $\psi(\tau)$ induced by the BSDE with jump associated to delayed generator $f^{\delta}$. Roughly speaking, we have for each $t\in[0,T]$,    
\begin{eqnarray*}
\rho^{\delta,\psi,\tau}(t) = -X^{\delta,\psi, \tau }(t),
\end{eqnarray*}
where $X^{\delta,\psi,\tau}$ is the solution of the BSDE associated with the generator $f^{\delta}$, terminal condition $\psi(\tau)$ and terminal time $\tau$. We are in the context where a very persistent economic agent the worst case. For this reason, we require a risk measure which would be the supremum over $\delta$ of the family of risk measures $(\rho^{\delta,\psi,\tau}(\sigma))_{\delta\in\mathcal{A}}$ defined by
\begin{eqnarray*}
\rho^{\psi,\tau}(\sigma)=ess\sup_{\delta\in \mathcal{A}}\rho^{\delta,\psi,\tau}(\sigma) = ess \sup_{\delta \in \mathcal{A}}(-X^{\delta,\psi,\tau}(\sigma))= - ess \inf_{\delta \in \mathcal{A}}X^{\delta,\psi,\tau}(\sigma)	.
\end{eqnarray*}
Our aim in this section is to find at each stopping time $\sigma\in [0,\widehat{\sigma}]$ ($\widehat{\sigma}$ is a stopping time defined such that we can apply the comparison principe for BSDE with jump and delayed generator), the stopping time  $\bar{\tau}\in [\sigma,T]$ which minimizes $\rho^{\psi,\tau}(\sigma)$ the risk measure of our persistent agent. To resolve this problem, let consider the value function $u$ is defined by: 
\begin{eqnarray}\label{Ambi}
u(\sigma)= ess \inf_{\tau \in [\sigma,T]} ess \sup_{\delta \in \mathcal{A}}\rho^{\delta,\psi,\tau}(\sigma).
\end{eqnarray}
On the other hand, and for a given $\sigma\in [0,\widehat{\sigma}]$, let us consider the two value function:
\begin{eqnarray}\label{M}
\overline{V}(\sigma) = ess \inf_{\delta \in \mathcal{A}} ess \sup_{\tau \in [\sigma,T]}X^{\delta,\psi,\tau}(\sigma)
\end{eqnarray}
and
\begin{eqnarray} \label{N}
\underline{V}(\sigma) = ess \sup_{\tau \in [\sigma,T]} ess \inf_{\delta \in \mathcal{A}}X^{\delta,\psi,\tau}(\sigma).
\end{eqnarray}
\begin{remark}
It not difficult to derive that $\underline{V}(\sigma)= -u(\sigma)$ a.s.
\end{remark}
Let us give this definition which permit us to understand condition of solvability to our problem.
\begin{definition}
Let $\sigma$ be in $[0,\widehat{\sigma}]$. A pair $(\overline{\tau}, \overline{\delta})\in [\sigma,T]\times\mathcal{A}$ is called a $\sigma$-saddle point of our problem \eqref{M} or \eqref{N} if
\begin{itemize}
\item [(i)] $\underline{V}(\sigma) = \overline{V}(\sigma)$ a.s. 
\item [(ii)] the essential infimum in \eqref{M} is attained at $\overline{\delta}$.
\item [(iii)] the essential supremum in \eqref{N} is attained at $\overline{\tau}$. 
\end{itemize}
\end{definition}
\begin{remark} 
\item [(i)] It is not difficult to prove that for each $\sigma\in [0,\widehat{\sigma}]$, $(\overline{\tau},\overline{\delta})$ is a $\sigma$-saddle point if and only if for each $(\tau,\delta)\in [\sigma,T]\times\mathcal{A}$, we have
\begin{eqnarray*}
X^{\overline{\delta},\psi,\tau}(\sigma)\leq X^{\overline{\delta},\psi,\overline{\tau}}(\sigma)\leq X^{\delta,\psi,\overline{\tau}}(\sigma),\;\;\; a.s.	
\end{eqnarray*}
\item [(ii)] For each $\sigma\in [0,\widehat{\sigma}]$, if $(\overline{\delta},\overline{\tau})$ is a $\sigma$-sadle point, then $\overline{\delta}$ and $\overline{\tau}$ attain respectively the infimum and the supremum in $\underline{V}(\sigma)$ that is
\begin{eqnarray*}
\underline{V}(\sigma)=ess\sup_{\tau \in [\sigma,T]} ess\inf_{\delta \in \mathcal{A}}X^{\delta,\psi,\tau}(\sigma)=ess \inf_{\delta \in \mathcal{A}}X^{\delta,\psi,\overline{\tau}}(\sigma)=X^{\overline{\delta},\psi,\overline{\tau}}(\sigma)
\end{eqnarray*}
Hence, $\overline{\tau}$ is an optimal stopping time for the agent who wants to minimize over stopping times her risk-measure at time $\sigma$ under ambiguity (see \eqref{Ambi}).
Also, since $\overline{\delta}$ attains the essential infimum in \eqref{M}, $\overline{\delta}$ corresponds at time $\sigma$ to a worst case scenario. Hence, the robust optimal stopping problem \eqref{Ambi} reduces to a classical optimal stopping problem associated with a worst-case scenario among the possible ambiguity parameters $\delta\in \mathcal{A}$.
\end{remark}
 Since for all $\sigma\in [0,\widehat{\sigma}]$, we have clearly $\underline{V}(\sigma)\leq \overline{V}(\sigma)$ a.s., we want to determine when the equality holds, characterize the value function and address the question of existence of a $\sigma$- saddle point.
 
 For this purpose, let us relate the game problem to the optimization problem for RBSDEs stated previously. Let consider $(Y^{\delta},Z^{\delta},U^{\delta}(.))$ the solution of the reflected BSDE with jump and delayed generator $(\psi(\tau),f^{\delta},\psi)$. According to section 4, there exist a stopping time $\sigma_{\delta}$ defined as in \eqref{ST} such that, for each $\sigma\in [0,\sigma_{\delta}]$, we have 
\begin{eqnarray*}
Y^{\delta}(\sigma) = ess \sup_{\tau \in [\sigma_{\delta},T]} X^{\delta,\psi,\tau}(\sigma),\;\;\;\mbox{a.s}.	
\end{eqnarray*} 
Next, applied comparison theorem to the family of reflected BSDE with jump and delayed generator $(\psi(\tau),f^{\delta},\psi)$, there exists  a stopping time $\widehat{\sigma}$ defined by \eqref{Compopti} such that for each $\sigma \in [0,\widehat{\sigma}]$,
\begin{eqnarray*}
\overline{V}(\mathcal{S}) = ess \inf_{\delta\in \mathcal{A}}Y^{\delta}(\sigma),\;\;\;\; \mbox{a.s}.
\end{eqnarray*}
Let set $f=\inf_{\delta\in \mathcal{A}}$ and consider $(Y,Z,U(.))$ as the solution of the reflected BSDE $(\psi(\tau),f,\psi)$.

\begin{theorem}\label{Topt}
Suppose that $f^{\delta},\;f$ satisfy assumptions $({\bf A3})$ and $({\bf A4})$ for all  $\delta\in\mathcal{A}$. Suppose also that there exist $\overline{\delta}\in \mathcal{A}$ such that $f=f^{\overline{\delta}}$. Then, there exists a value function, which is characterized as the solution of the reflected BSDE $(\psi(\tau),f,\psi)$, that is, for each $\sigma \in[0,\widehat{\sigma}]$, we have
\begin{eqnarray*}
Y(\sigma)=\underline{V}(\sigma)=\overline{V}(\sigma)\;\;\; \mbox{a.s}.
\end{eqnarray*}
Moreover, the minimal risk measure, defined by \eqref{Ambi}, verifies, for each $\sigma\in[0,\widehat{\sigma}]$, $u(\sigma)=-Y(\sigma)$, a.s.
\end{theorem}
\begin{proof}
The proof follows the same approach as one of Theorem 5.3 appear in \cite{10}. Except the fact that we deal in the stochastic interval $[0,\widehat{\sigma}]$ where $\widehat{\sigma}$ is defined by \eqref{Compopti}. This is due to the use of the comparison theorem which is valid only on this type of interval.
\end{proof}
We have this result which generalize Corollary 5.4 in \cite{10} to BSDE with jump and delayed generator. 
\begin{corollary}
Suppose assumptions of Theorem \ref{Topt} be satisfied and the obstacle $psi$ be l.u.s.c. along stopping times. Let $\widehat{\sigma}$ be a stopping time defined by \eqref{Compopti}. For each $\sigma\in [0,\widehat{\sigma}]$, we set
\begin{eqnarray*}
\tau^{*}=\inf \{s\leq \sigma,\;\; Y(u)=\psi(u)\}.
\end{eqnarray*}
Then, $(\tau^{*}_{\sigma},\overline{\sigma})$ is an $\sigma$-saddle point, that is $Y(\sigma) = X^{\overline{\sigma},Y^{\delta}_{\tau^{*}_{\sigma}}}(\sigma)$ a.s. In other word, $\tau_{\sigma}^{*}$ is an optimal stopping time for the agent who wants to minimize her risk measure at time $\sigma$ and $\overline{\delta}$ corresponds to a worst scenario.
\end{corollary}
Let end this paper with this remark in order to summarize the rest of our generalization.
\begin{remark}
Using the same approach it not difficult to respectively establish the analog of Proposition 5.5, Theorem 5.6 and Corollary 5.7 of \cite{10}.
\end{remark}

\end{document}